\documentclass[12pt]{amsart}

\usepackage{graphicx}
\usepackage{subfigure}
\usepackage{eepic}
\usepackage{xcolor}
\selectcolormodel{gray}
\usepackage{tikz}
\usepackage{amsmath}
\usepackage{a4wide}
\usepackage[utf8]{inputenc}
\usepackage{amssymb}
\usepackage{amsopn}
\usepackage{epsfig}
\usepackage{amsfonts}
\usepackage{latexsym}
\usepackage{amsthm}
\usepackage{enumerate}
\usepackage[UKenglish]{babel}
\usepackage{verbatim}
\usepackage{color}

\allowdisplaybreaks

\newtheorem{theorem}{Theorem}[section]
\newtheorem{lemma}[theorem]{Lemma}
\newtheorem{proposition}[theorem]{Proposition}

\theoremstyle{definition}

\theoremstyle{remark}

\numberwithin{equation}{section}

\newcommand{\R}{\ensuremath{\mathbb{R}}}

\newcommand{\N}{\ensuremath{\mathbb{N}}}

\renewcommand{\H}{\mathcal{H}}

\newcommand{\D}{\mathcal{D}}

\begin{document}

\title[Quantitative recurrence properties for self-conformal sets]{Quantitative recurrence properties for self-conformal sets}

\author{Simon Baker and Michael Farmer}
\address{Simon Baker: School of Mathematics, University of Birmingham, Birmingham,  B15 2TT, UK}
\email{simonbaker412@gmail.com}
\address{Michael Farmer: Mathematics institute, University of Warwick, Coventry, CV4 7AL, UK}
\email{michaelfarmer868@gmail.com }

\date{\today}

\subjclass[2010]{}

\begin{abstract}
In this paper we study the quantitative recurrence properties of self-conformal sets $X$ equipped with the map $T:X\to X$ induced by the left shift. In particular, given a function $\varphi:\N\to(0,\infty),$ we study the metric properties of the set $$R(T,\varphi)=\left\{x\in X:|T^nx-x|<\varphi(n)\textrm{ for infinitely many }n\in \N\right\}.$$ Our main result shows that for the natural measure supported on $X$, $R(T,\varphi)$ has zero measure if a natural volume sum converges, and under the open set condition $R(T,\varphi)$ has full measure if this volume sum diverges.

\end{abstract}

\maketitle

\section{Introduction}\label{sec:1}
The notion of recurrence is of central importance within Dynamical Systems and Ergodic Theory. A well known theorem due to Poincar\'{e} states that if $(X,\mathcal{B},\mu)$ is a probability space, and $T:X\to X$ is a measure preserving transformation, then for any $E\in \mathcal{B}$ we have $$\mu\left(\{x\in E:T^nx\in E \textrm{ for infinitely many }n\in \N\}\right)=\mu(E).$$ If $X$ is endowed with a metric $d$ so that $(X,d)$ is a separable metric space, and $\mathcal{B}$ is the Borel $\sigma$-algebra, then Poincar\'{e}'s theorem allows us to conclude the following topological statement: 
\begin{equation}
\label{Topological reccurrence}
\liminf_{n\to\infty}d(T^nx,x)=0
\end{equation}for $\mu$-almost every $x\in X$. We call $(X,\mathcal{B},\mu,d,T)$ a metric measure-preserving system or an m.m.p.s. The information provided by \eqref{Topological reccurrence} is qualitative in nature. It tells us nothing about the speed at which an orbit can recur upon its initial point. One of the first general quantitative recurrence results was proved by Boshernitzan in \cite{Bos}. 
\begin{theorem}[\cite{Bos}]
	Let $(X,\mathcal{B},\mu,d,T)$ be a m.m.p.s. Assume that for some $\alpha>0$ the $\alpha$-dimensional Hausdorff measure $\mathcal{H}^{\alpha}$ is $\sigma$-finite on $(X,d)$. Then for $\mu$-almost every $x\in X$ we have $$\liminf_{n\to\infty} n^{1/\alpha}d(T^nx,x)<\infty.$$ Moreover, if $\mathcal{H}^{\alpha}(X)=0$ then for $\mu$-a.e $x\in X$ $$\liminf_{n\to\infty} n^{1/\alpha}d(T^nx,x)=0.$$
\end{theorem}
Building upon the work of Boshernitzan, Barreira and Saussol in \cite{BarSol} showed how the lower local dimension of a measure can be used to obtain quantitative recurrence results. 

\begin{theorem}[\cite{BarSol}]
If $T:X\to X$ is a Borel measurable map on $X\subset \R^d$, and $\mu$ is a $T$-invariant Borel probability measure on $X,$ then for $\mu$-almost every $x\in X,$ we have 
	$$\liminf_{n\to\infty}n^{1/\alpha}d(T^nx,x)=0\textrm{ for any }\alpha>\liminf_{r\to 0}\frac{\log \mu(B(x,r))}{\log r}.$$
\end{theorem} 

A suitable framework for describing recurrence quantitatively is the following. Given $(X,\mathcal{B},\mu,d,T)$ a m.m.p.s. and $\varphi:\N\times X\to (0,\infty)$, let $$R(T,\varphi):=\left\{x\in X:d(T^nx,x)<\varphi(n,x) \textrm{ for infinitely many }n\in \N\right\}.$$ Typically one is interested in determining the metric properties of $R(T,\varphi)$ and relating these properties to $T$ and $\varphi$. This was the line of research pursued by Tan and Wang in \cite{TanWang} where they calculated the Hausdorff dimension of $R(T,\varphi)$ when $T$ is the $\beta$-transformation. Later Seuret and Wang proved similar results for self-conformal sets in \cite{SeuretWang}. As remarked upon by Chang et al in \cite{ChangWuWu}, very few results exist on the Hausdorff measure of $R(T,\varphi)$. In this paper we continue the line of research instigated in \cite{ChangWuWu} and obtain results on the Hausdorff measure of $R(T,\varphi)$ when $X$ is a self-conformal set and $T$ is the natural map induced by the left shift. Before introducing our problem formally, we would like to mention a related topic and include some references.  

The shrinking target problem is concerned with determining the speed at which the orbit of a $\mu$-typical point accumulates on a fixed point $x_0$. The shrinking target problem and the problem of obtaining quantitative recurrence results have many common features. One can define a suitable analogue of the set $R(T,\varphi)$ and ask what are its metric properties. For the shrinking target problem much more is known about the Hausdorff measure of this set, see \cite{CheKle,FMP}, for results on the Hausdorff dimension of this set see \cite{HV,LWWX,WW}.

\subsection{Statement of results}
Let $V\subset \R^d$ be an open set, a $C^{1}$ map $\phi:V\to \R^{d}$ is a conformal mapping if it preserves angles. Equivalently $\phi$ is a conformal mapping if the differential $\phi'$ satisfies $|\phi'(x)y|=|\phi'(x)||y|$ for all $x\in V$ and $y\in \R^{d}$. Let $\Phi=\{\phi_{i}\}_{i\in \D}$ be a finite set of contractions on a compact set $Y\subset \R^d$, i.e. there exists $r\in(0,1)$ such that $|\phi_i(x)-\phi_i(y)|\leq r|x-y|$ for all $x,y\in Y$. We say that $\Phi$ is a conformal iterated function system if each $\phi_i$ can be extended to an injective conformal contraction on some open connected neighbourhood $V$ that contains $Y$ and $0<\inf_{x\in V} |\phi_i'(x)|\leq \sup_{x\in V}|\phi_i'(x)|<1$. Throughout this paper we will assume that the differentials are H\"{o}lder continuous. This means there exists $\alpha>0$ and $c>0$ such that $$||\phi_i'(x)|-|\phi_i'(y)||\leq c|x-y|^\alpha$$ for all $x,y\in V$. A well known result due to Hutchinson \cite{Hut} implies that for any conformal iterated function system there exists a unique non-empty compact set $X\subset \R^d$ such that 
$$X=\bigcup_{i\in \D}\phi_{i}(X).$$ We call the set $X$ the self-conformal set of $\Phi$. 


In what follows, if $I=(i_1,\ldots,i_n)$ then we let $\phi_{I}=\phi_{i_1}\circ \cdots \circ \phi_{i_n},$ $X_{I}=\phi_{I}(X), \|\phi_{I}'\|=\sup_{x\in V}|\phi_{I}'(x)|,$ and $|I|$ will denote the length of $I$. We will refer to the set $X_{I}$ as a cylinder or a cylinder set. For a word $I\in \cup_n\D^n$ we let $I^{\infty}$ denote the element of $\D^{\mathbb{N}}$ obtained by repeating $I$ indefinitely. Similarly for $k\geq 1$ we let $I^k$ denote the word $I$ repeated $k$ times. 

Given a conformal iterated function system we denote by $\gamma$ the unique value satisfying $P(\gamma)=0,$ where $$P(s):=\lim_{n\to\infty}\frac{1}{n}\log \sum_{I\in \D^{n}}\|\phi_{I}'\|^{s}.$$
For a proof of the existence and uniqueness of $\gamma$ see \cite{Fal}. We say that a conformal iterated function system $\Phi$ satisfies the open set condition if there exists an open set $O\subset \mathbb{R}^d$ such that $\phi_{i}(O)\subseteq O$ for all $i\in \D$ and $\phi_{i}(O)\cap \phi_{j}(O)=\emptyset$ for $i\neq j$. Under the open set condition, the Hausdorff dimension of the self-conformal set $X$ is equal to $\gamma$ and $\mathcal{H}^{\gamma}(X)$ is positive and finite. Moreover, under the open set condition $X$ is Ahlfors regular, this means that there exists $C>1$ such that \begin{equation}
\label{Ahlfors regular}\frac{r^{\gamma}}{C}\leq\mathcal{H}^{\gamma}(X\cap B(x,r))\leq Cr^{\gamma}
\end{equation} for all $x\in X$ and $0<r<Diam(X)$. These results are well known and date back to the work of Ruelle \cite{Rue}. For a proof see \cite{Fal}.


One can encode elements of a self-conformal set using sequences in $\D^{\N}$ as follows. Let $\pi:\D^{\N}\to X$ be given by $$\pi((i_m))=\lim_{n\to\infty}(\phi_{i_1}\circ\cdots \circ \phi_{i_n})(0).$$ The map $\pi$ is surjective. Moreover, equipping $\D^{\N}$ with the product topology it can be shown that $\pi$ is continuous. For $x\in X$, we call any sequence $(i_m)\in \D^{\mathbb{N}}$ such that $\pi((i_m))=x$ a coding of $x$. Without any separation hypothesis on the conformal iterated function system, it is possible that a typical $x\in X$ will have multiple, possibly infinitely many, distinct codings. However, assuming the open set condition, $\mathcal{H}^{\gamma}$-almost every $x\in X$ has a unique coding. This follows from Theorem 3.7. from \cite{Kae}. With these observations in mind we now define our map $T:X\to X$ induced by the left shift on $\D^{\mathbb{N}}$. Let 
$Tx=\pi((i_{m+1}))$ where $(i_m)$ is an arbitrary choice of coding for $x$. Since under the open set condition $\mathcal{H}^{\gamma}$-almost every $x\in X$ has a unique coding, it follows from the definition of $T$ that under this assumption $T^{n}x=\pi((i_{m+n}))$ for $\mathcal{H}^{\gamma}$-almost every $x\in X$ for any $n\in \N$. We will only be interested in statements which hold for $\mathcal{H}^{\gamma}$-almost every $x\in X.$ As such when we assume the open set condition we can effectively ignore those points with multiple codings and assume that $T$ maps $x$ to the point whose coding is the unique coding of $x$ with the first digit removed. 

Recalling the definition of $R(T,\varphi)$ from our introduction, and taking $d$ to be the usual Euclidean metric, we may now state our main result.

\begin{theorem}
	\label{Main theorem}
	Let $\Phi$ be a conformal iterated function system and $\varphi:\mathbb{N}\to (0,\infty)$. Then the following statements are true:
\begin{enumerate}
	\item If $\sum_{n=1}^{\infty}\varphi(n)^{\gamma}<\infty$ then $\mathcal{H}^{\gamma}(R(T,\varphi))=0$.
	\item If $\Phi$ satisfies the open set condition and $\sum_{n=1}^{\infty}\varphi(n)^{\gamma}=\infty,$ then $\mathcal{H}^{\gamma}(R(T,\varphi))=\mathcal{H}^{\gamma}(X)$.
\end{enumerate}
\end{theorem}
Theorem \ref{Main theorem} was proved in \cite{ChangWuWu} by Chang et al for homogeneous self-similar sets in $\R$ satisfying the strong separation condition (i.e. $\phi_{i}(X)\cap \phi_{j}(X)=\emptyset$ for all $i\neq j$). As such Theorem \ref{Main theorem} significantly improves upon \cite{ChangWuWu} as it allows for a more general class of iterated function systems and has weaker separation hypothesis. \\

\noindent \textbf{Notation.} Given two positive real valued functions $f$ and $g$ defined on some set $S,$ we write $f\preceq g$ if there exists a positive constant $C$ such that $f(x)\leq C g(x)$ for all $x\in S$. Similarly we write $f\succeq g$ if $g\preceq f$. We write $f\asymp g$ if  $f\preceq g$ and $f\succeq g.$

\section{Proof of Theorem \ref{Main theorem}}
Before diving into our proof it is useful to recall some well known properties of self-conformal sets. We start by stating some properties that hold without any separation assumptions on our conformal iterated function system. These will be used in our proof of the convergence part of Theorem \ref{Main theorem}.

The following properties hold for any conformal iterated function system:
\begin{itemize}
	\item For any $n\in \N$ we have 
	\begin{equation}
	\label{Sum of derivatives}
	\sum_{I\in \D^n}\|\phi_{I}'\|^{\gamma}\asymp 1.
	\end{equation}
	\item Let $I\in \cup_n \D^n$. Then for any $x,y\in V$ we have 
	\begin{equation}
	\label{distance distortion}
	|\phi_{I}(x)-\phi_{I}(y)|\asymp \|\phi_{I}'\||x-y|.
	\end{equation}
\item Let $x\in X$ and $(i_m)\in \D^{\N}$ be a coding of $x$. For any $0<r<Diam(X)$ there exists $N\in\mathbb{N}$ such that \begin{equation}
	\label{cylinder approx}
	X_{i_{1},\ldots,i_{N}}\subset B(x,r) \textrm{ and } Diam(X_{i_{1},\ldots,i_{N}})\asymp r.
	\end{equation}
\end{itemize} Equation \eqref{Sum of derivatives} is essentially a consequence of the fact that for each $I\in \cup_n \D^n,$ the quantity $\|\phi_{I}'\|^{\gamma}$ is comparable to the mass a suitably defined Gibbs probability measure on $\D^{\N}$ assigns to the cylinder corresponding to the word $I$ (see \cite{Fal}).  For a proof of \eqref{distance distortion} see Lemma 6.1 from \cite{AKT}. The proof of \eqref{cylinder approx} is standard.

Now suppose $\Phi$ is a conformal iterated function system satisfying the open set condition. For any $I\in \cup_n \D^n$ we let 
$$\widetilde{X_{I}}:=\{x\in X_I: x \textrm{ has a unique coding }\}.$$ Since $\mathcal{H}^{\gamma}$-almost every $x\in X$ has a unique coding, we have
\begin{equation}
	\label{Samemeasure}
	\mathcal{H}^{\gamma}(X_I)=\mathcal{H}^{\gamma}(\widetilde{X}_I)
\end{equation}for any $I\in \cup_n\D^n$.
Let $\mu:=\mathcal{H}^{\gamma}|_{X}$ be the $\gamma$-dimensional Hausdorff measure restricted to $X.$ The properties stated below are well known for cylinders without the unique coding restriction. These properties still hold for the sets $\widetilde{X_{I}}$ because of \eqref{Samemeasure}. 

The following properties hold:
\begin{itemize}
	\item For any $n\in\mathbb{N}$ and $I,J\in \D^{n}$ such that $I\neq J,$ we have
	\begin{equation}
		\label{measure zero intersection}
		\mu(\widetilde{X}_{I}\cap \widetilde{X}_{J})=0.
	\end{equation}
	\item For any $I,J\in \cup_n\D^n$
	\begin{equation}
		\label{Weak Bernoulli measure}
		\mu(\widetilde{X}_{IJ})\asymp \mu(\widetilde{X}_{I})\mu(\widetilde{X}_{J}).
	\end{equation}
	\item For any $I\in \cup_n\D^n$
	\begin{equation}
		\label{Measure and diameter}
		\mu(\widetilde{X}_{I})\asymp Diam(X_{I})^{\gamma}.
	\end{equation}
	\item There exists $\kappa\in(0,1)$ such that for any $I\in \cup_n \D^n$
	\begin{equation}
		\label{measure decay}
		\mu(\widetilde{X}_{I})\preceq \kappa^{|I|}.
	\end{equation}
		\item For any $n\in \N$ we have 
	\begin{equation}
		\label{Sum prefixes}
		\sum_{I\in \D^n}\mu(\widetilde{X}_I)=\mu(X).
	\end{equation}
	Similarly, for any $J\in \cup_{n}\D^n$ and $n\geq |J|$ we have \begin{equation}
		\label{Sum prefixes2}\sum_{\stackrel{I\in \D^n}{J\textrm{ is a prefix of }I}}\mu(\widetilde{X}_{I})=\mu(\widetilde{X}_J).
	\end{equation}
	
		\end{itemize}
In the above we have denoted the concatenation of two words $I$ and $J$ by $IJ$. Property \eqref{measure zero intersection} follows from Theorem 3.7. from \cite{Kae}. For a proof of the remaining properties see \cite{Fal} and \cite{Rue}. Properties \eqref{Weak Bernoulli measure}, \eqref{Measure and diameter}, and \eqref{measure decay} are essentially a consequence of the fact that under the open set condition $\mu$ is equivalent to the pushforward of the aforementioned Gibbs probability measure defined on $\D^{\N}$. Properties \eqref{Sum prefixes} and \eqref{Sum prefixes2} are a consequence of \eqref{measure zero intersection} and the fact $X=\cup_{I\in \D^n}\phi_{I}(X)$ for any $n\in \N$.
\subsection{Proof of Theorem \ref{Main theorem}.1. (Convergence part)}
The proof of the convergence part of Theorem \ref{Main theorem} will be a consequence of the following lemma.


\begin{lemma}
	\label{covering lemma}
There exists $K>0$ such that for $n$ sufficiently large 
$$\{x\in X:|T^nx-x|<\varphi(n)\}\subseteq \bigcup_{I\in \D^n}B(\pi(I^{\infty}),K\|\phi_{I}'\|\varphi(n)).$$
\end{lemma}
\begin{proof}
Let us fix $x$ such that $|T^nx-x|<\varphi(n).$ Let $I\in \D^n$ be such that $T^{n}(x)=\phi_{I}^{-1}(x).$ Such an $I$ exists by the definition of $T$ and the coding map $\pi$. Using the fact $\pi(I^{\infty})$ is fixed under $\phi_{I}^{-1}$ together with the triangle inequality we have 
$$|x-\pi(I^{\infty})|\geq |T^n(x)-\pi(I^{\infty})|-|T^n(x)-x|>|\phi_{I}^{-1}(x)-\phi_{I}^{-1}(\pi(I^{\infty}))|-\varphi(n).$$ Applying \eqref{distance distortion} it follows that there exists $K'>0$ such that 
$$|x-\pi(I^{\infty})|> K'\|\phi_{I}'\|^{-1}|x-\pi(I^{\infty})|-\varphi(n).$$ For $n$ sufficiently large $K'\|\phi_{I}'\|^{-1}-1$ is positive and therefore this expression can be rearranged to give 
\begin{equation}
\label{easy}
|x-\pi(I^{\infty})|<\frac{\varphi(n)}{K'\|\phi_{I}'\|^{-1}-1}.
\end{equation} Since each $\phi_i$ is strictly contracting, it follows that $K'\|\phi_{I}'\|^{-1}-1\asymp \|\phi_{I}'\|^{-1}$ for $n$ sufficiently large. This fact together with \eqref{easy} implies our result.
\end{proof}

\begin{proof}[Proof of Theorem \ref{Main theorem}.1.]
Assume $\varphi$ is such that $\sum_{n=1}^{\infty}\varphi(n)^{\gamma}<\infty.$ By Lemma \ref{covering lemma} and the definition of Hausdorff measure (see \cite{Fal2}) we have the following: 
\begin{align*}
\H^{\gamma}(R(T,\varphi))&\leq \liminf_{N\to\infty}\sum_{n=N}^{\infty}\sum_{I\in\D^n}Diam(B(\pi(I^{\infty}), K\|\phi_{I}'\|\varphi(n))^{\gamma}\\
&\preceq \liminf_{N\to\infty}\sum_{n=N}^{\infty}\sum_{I\in\D^n}(\|\phi_{I}'\|\varphi(n)))^\gamma\\
&=\liminf_{N\to\infty}\sum_{n=N}^{\infty}\varphi(n)^\gamma\sum_{I\in\D^n}\|\phi_{I}'\|^{\gamma}\\
&\stackrel{\eqref{Sum of derivatives}}{\preceq}\liminf_{N\to\infty}\sum_{n=N}^{\infty}\varphi(n)^\gamma\\
&=0.
\end{align*} In the last line we used our assumption $\sum_{n=1}^{\infty}\varphi(n)^\gamma<\infty.$ 
\end{proof}
\subsection{Proof of Theorem \ref{Main theorem}.2. (Divergence part)}
Our proof of the divergence part of Theorem \ref{Main theorem} is based upon the proof of Theorem 1.4. from \cite{Bak}. We will make use of the following two lemmas.

\begin{lemma}
	\label{Density lemma}
	Let $X$ be a compact set in $\R^{d}$ and let $\mu$ be a finite doubling measure on X such
	that any open set is $\mu$-measurable. Let $E$ be a Borel subset of $X$. Assume that there are
	constants $r_0,c > 0$ such that for any ball $B$ with radius less than $r_0$ and centre in $X$ we have
	$$\mu(E \cap B) > c \mu(B).$$
	Then $\mu(X \setminus E) = 0.$
\end{lemma}For a proof of Lemma \ref{Density lemma} see \cite[\S 8]{BDV}. Note that a measure $\mu$ supported on a compact set $X$ is doubling if there exists a constant $C>1$ such that for any $x\in X$ and $r>0$ we have $$\mu(B(x,2r))\leq C \mu(B(x,r)).$$ Since $X$ is Ahlfors regular it follows from \eqref{Ahlfors regular} that $\mu$ is automatically a doubling measure.
\begin{lemma}
\label{Erdos lemma}
Let $X$ be a compact set in $\R^{d}$ and let $\mu$ be a finite measure on $X$. Also, let $E_n$ be
a sequence of $\mu$-measurable sets such that $\sum_{n=1}^{\infty}\mu(E_n)=\infty.$ Then
$$\mu(\limsup_{n\to\infty} E_{n})\geq \limsup_{Q\to\infty}\frac{(\sum_{n=1}^{Q}\mu(E_{n}))^{2}}{\sum_{n,m=1}^{Q}\mu(E_{n}\cap E_m)}.$$
\end{lemma}For a proof of Lemma \ref{Erdos lemma} see \cite[Lemma 5]{Spr}.

We may now proceed with our proof of the divergence part of Theorem \ref{Main theorem}. Let $I=(i_1,\ldots,i_n)\in\D^n$ and consider the ball $B(\pi(I^{\infty}),\varphi(n)/2)$. Recall that given a finite word $I$ we let $I^{\infty}$ denotes the element of $\D^{\N}$ obtained by repeating $I$ indefinitely, and $I^{k}$ denotes the word $I$ repeated $k$ times. Applying \eqref{cylinder approx} we know that there exists $k_{I}\geq 0$ and $1\leq s_{I}\leq n-1$ such that $$X_{I^{k_I}(i_1,\ldots,i_{s_I})}\subseteq B(\pi(I^{\infty}),\varphi(n)/2)$$ and 
\begin{equation}
\label{inclusion}
Diam(X_{I^{k_I}(i_1,\ldots,i_{s_I})})\asymp \frac{\varphi(n)}{2}.
\end{equation} Now consider the set $\widetilde{X}_{I^{k_I+1}(i_1,\ldots,i_{s_I})}.$ For any $x\in \widetilde{X}_{I^{k_I+1}(i_1,\ldots,i_{s_I})}$ we have $T^nx\in \widetilde{X}_{I^{k_I}(i_1,\ldots,i_{s_I})}.$ Moreover, since $$\widetilde{X}_{I^{k_I+1}(i_1,\ldots,i_{s_I})}\subseteq X_{I^{k_I}(i_1,\ldots,i_{s_I})}\subseteq  B(\pi(I^{\infty}),\varphi(n)/2),$$ we may conclude by the triangle inequality that if $x\in \widetilde{X}_{I^{k_I+1}(i_1,\ldots,i_{s_I})}$ then $|T^nx-x|<\varphi(n)$. So if we let $$E_n'=\bigcup_{I\in \D^n}\widetilde{X}_{I^{k_I+1}(i_1,\ldots,i_{s_I})}$$ then $$\limsup_{n\to\infty} E_n'\subseteq R(T,\varphi).$$ To prove the divergence part of Theorem \ref{Main theorem} it suffices to show that $\mu(\limsup_{n\to\infty} E_n')=\mu(X).$ To do this we will apply Lemma \ref{Density lemma}. As such let us fix an arbitrary ball $B$ with centre in $X$ and radius less then $Diam(X)$. Applying \eqref{cylinder approx} we know that there exists $J\in \cup_n D^n$ such that $X_{J}\subseteq B$ and $Diam(X_J)\asymp Radius(B)$. By \eqref{Ahlfors regular} and \eqref{Measure and diameter} we know that $\mu(X_J)\asymp \mu(B).$ Therefore to prove the divergence part of Theorem \ref{Main theorem}, instead of proving that there exists $c>0$ such that $\mu(\limsup_{n\to\infty} E_n' \cap B)>c\mu(B)$, it suffices to show that there exists $c>0$ such that $\mu(\limsup_{n\to\infty} E_n' \cap X_J)>c\mu(X_J).$ 

For $n\geq |J|$ if we let $$E_n=\bigcup_{\stackrel{I\in \D^n}{J\textrm{ is a prefix of }I}}\widetilde{X}_{I^{k_I+1}(i_1,\ldots,i_{s_I})}$$ then $\limsup_{n\to\infty} E_n\subset \limsup_{n\to\infty} E_n'\cap X_J$. Therefore to prove the divergence part of Theorem \ref{Main theorem} it suffices to show that there exists $c>0$ such that 
\begin{equation}
\label{Sufficestoshow}
\limsup_{n\to\infty} E_n>c\mu(X_J).
\end{equation}
To do this we will use Lemma \ref{Erdos lemma}. To use this lemma we first have to check $\sum_{n=|J|}^{\infty}\mu(E_n)=\infty.$

\begin{lemma}
	\label{Divergence lemma}
	We have $\sum_{n=|J|}^{Q}\mu(E_n)\asymp \mu(X_J)\sum_{n=|J|}^{Q}\varphi(n)^{\gamma}.$
\end{lemma}
\begin{proof}The following holds
\begin{align*}
\sum_{n=|J|}^{Q}\mu(E_n)&=\sum_{n=|J|}^{Q}\mu\left(\bigcup_{\stackrel{I\in \D^n}{J\textrm{ is a prefix of }I}}\widetilde{X}_{I^{k_I+1}(i_1,\ldots,i_{s_I})}\right)\\
&\stackrel{\eqref{measure zero intersection}}{=}\sum_{n=|J|}^{Q}\sum_{\stackrel{I\in \D^n}{J\textrm{ is a prefix of }I}}\mu(\widetilde{X}_{I^{k_I+1}(i_1,\ldots,i_{s_I})})\\
	&\stackrel{\eqref{Weak Bernoulli measure}}{\asymp}\sum_{n=|J|}^{Q}\sum_{\stackrel{I\in \D^n}{J\textrm{ is a prefix of }I}}\mu(\widetilde{X}_I)\mu(\widetilde{X}_{I^{k_I}(i_1,\ldots,i_{s_I})})\\
	&\stackrel{\eqref{Measure and diameter},\eqref{inclusion}}{\asymp}\sum_{n=|J|}^{Q}\varphi(n)^{\gamma}\sum_{\stackrel{I\in \D^n}{J\textrm{ is a prefix of }I}}\mu(\widetilde{X}_I)\\
	&\stackrel{\eqref{Sum prefixes2}}{=} \mu(\widetilde{X}_J)\sum_{n=|J|}^{Q}\varphi(n)^{\gamma}\\
	&\stackrel{\eqref{Samemeasure}}{=} \mu(X_J)\sum_{n=|J|}^{Q}\varphi(n)^{\gamma}.
\end{align*} 
\end{proof}
Lemma \ref{Divergence lemma} shows that when $\sum_{n=1}^{\infty}\varphi(n)^{\gamma}=\infty$ the sequence $(E_n)$ satisfies the hypothesis of Lemma \ref{Erdos lemma}. It will also be used in some of our later calculations. 

\begin{lemma}
	\label{Near independence}
	Let $I\in \D^n$ be such that $J$ is a prefix of $I.$ Then for $m>n$ we have 
	\begin{align*}
	\mu(\widetilde{X}_{I^{k_{I}+1}(i_1,\ldots,i_{s_I})}\cap E_m)&\preceq \mu(\widetilde{X}_{J})\mu(\widetilde{X}_{i_{|J|+1},\ldots,i_n})\kappa^{m-n}\varphi(m)^{\gamma}\\
	&+\mu(\widetilde{X}_{J})\mu(\widetilde{X}_{i_{|J|+1},\ldots,i_n}))\varphi(m)^{\gamma}\varphi(n)^{\gamma}.
	\end{align*}
\end{lemma}

\begin{proof}
Fix $I\in D^n$ such that $J$ is a prefix of $I$ and let $m>n$. It is useful to consider two separate cases. We first consider the case where $m\leq (k_{I}+1)n+s_{I}.$  

If $m\leq (k_{I}+1)n+s_{I}$ then there exists at most one $\tilde{I}\in \D^m$ such that $$\mu(\widetilde{X}_{I^{k_I+1}(i_1,\ldots,i_{s_I})}\cap \widetilde{X}_{\tilde{I}^{k_{\tilde{I}}+1}(\tilde{i}_1,\ldots,\tilde{i}_{s_{\tilde{I}}})})>0.$$ Therefore 
\begin{align*}
\mu(\widetilde{X}_{I^{k_I+1}(i_1,\ldots,i_{s_I})}\cap E_m)&=\mu(\widetilde{X}_{I^{k_I+1}(i_1,\ldots,i_{s_I})}\cap \widetilde{X}_{\tilde{I}^{k_{\tilde{I}}+1}(\tilde{i}_1,\ldots,\tilde{i}_{s_{\tilde{I}}})})\\
&\leq \mu( \widetilde{X}_{\tilde{I}^{k_{\tilde{I}}+1}(\tilde{i}_1,\ldots,\tilde{i}_{s_{\tilde{I}}})})\\
&\stackrel{\eqref{Weak Bernoulli measure}}{\asymp} \mu(\widetilde{X}_{\tilde{I}})\mu( \widetilde{X}_{\tilde{I}^{k_{\tilde{I}}}(\tilde{i}_1,\ldots,\tilde{i}_{s_{\tilde{I}}})})\\
&\stackrel{\eqref{Weak Bernoulli measure}}{\asymp} \mu(\widetilde{X}_{J})\mu(\widetilde{X}_{\tilde{i}_{|J|+1},\ldots,\tilde{i}_m})\mu( \widetilde{X}_{\tilde{I}^{k_{\tilde{I}}}(\tilde{i}_1,\ldots,\tilde{i}_{s_{\tilde{I}}})})\\
&\stackrel{\eqref{Weak Bernoulli measure}}{\asymp}\mu(\widetilde{X}_{J})\mu(\widetilde{X}_{i_{|J|+1},\ldots,i_n})\mu(\widetilde{X}_{\tilde{i}_{n+1},\ldots,\tilde{i}_m})\mu( \widetilde{X}_{\tilde{I}^{k_{\tilde{I}}}(\tilde{i}_1,\ldots,\tilde{i}_{s_{\tilde{I}}})})\\
&\stackrel{\eqref{measure decay}}{\preceq}\mu(\widetilde{X}_{J})\mu(\widetilde{X}_{i_{|J|+1},\ldots,i_n})\kappa^{m-n}\mu( \widetilde{X}_{\tilde{I}^{k_{\tilde{I}}}(\tilde{i}_1,\ldots,\tilde{i}_{s_{\tilde{I}}})})\\
&\stackrel{\eqref{Measure and diameter},\eqref{inclusion}}{\asymp} \mu(\widetilde{X}_{J})\mu(\widetilde{X}_{i_{|J|+1},\ldots,i_n})\kappa^{m-n}\varphi(m)^{\gamma}.
\end{align*}
Summarising the above, we have shown that if $n<m\leq (k_{I}+1)n+s_{I}$ then 
\begin{equation}
\label{Bound1}
\mu(\widetilde{X}_{I^{k_I+1}(i_1,\ldots,i_{s_I})}\cap E_m)\preceq \mu(\widetilde{X}_{J})\mu(\widetilde{X}_{i_{|J|+1},\ldots,i_n})\kappa^{m-n}\varphi(m)^{\gamma}.
\end{equation} Now suppose $m>(k_I+1)n+s_I.$ Then 
\begin{align*}
\mu(\widetilde{X}_{I^{k_I+1}(i_1,\ldots,i_{s_I})}\cap E_m)&=\sum_{\stackrel{\tilde{I}\in \D^m}{I^{k_I+1}(i_1,\ldots,i_{s_{I}})\textrm{ is a prefix of }\tilde{I}}}\mu(\widetilde{X}_{\tilde{I}^{k+1}(\tilde{i}_1,\ldots,\tilde{i}_{s_{\tilde{I}}})})\\
&\stackrel{\eqref{Weak Bernoulli measure}}{\asymp} \sum_{\stackrel{\tilde{I}\in \D^m}{I^{k_I+1}(i_1,\ldots,i_{s_{I}})\textrm{ is a prefix of }\tilde{I}}}\mu(\widetilde{X}_{\tilde{I}})\mu(\widetilde{X}_{\tilde{I}^{k}(\tilde{i}_1,\ldots,\tilde{i}_{s_{\tilde{I}}})})\\
&\stackrel{\eqref{Measure and diameter},\eqref{inclusion}}{\asymp} \sum_{\stackrel{\tilde{I}\in \D^m}{I^{k_I+1}(i_1,\ldots,i_{s_{I}})\textrm{ is a prefix of }\tilde{I}}}\mu(\widetilde{X}_{\tilde{I}})\varphi(m)^{\gamma}\\
&\stackrel{\eqref{Sum prefixes2}}{=}\mu(\widetilde{X}_{I^{k_I+1}(i_1,\ldots,i_{s_{I}})})\varphi(m)^{\gamma}\\
&\stackrel{\eqref{Weak Bernoulli measure}}{\asymp} \mu(\widetilde{X}_{I})\mu(\widetilde{X}_{I^{k_I}(i_1,\ldots,i_{s_{I}})})\varphi(m)^{\gamma}\\
&\stackrel{\eqref{Weak Bernoulli measure}}{\asymp} \mu(\widetilde{X}_{J})\mu(\widetilde{X}_{i_{|J|+1},\ldots,i_n})\mu(\widetilde{X}_{I^{k_I}(i_1,\ldots,i_{s_{I}})})\varphi(m)^{\gamma}\\
&\stackrel{\eqref{Measure and diameter},\eqref{inclusion}}{\asymp} \mu(\widetilde{X}_{J})\mu(\widetilde{X}_{i_{|J|+1},\ldots,i_n})\varphi(n)^{\gamma}\varphi(m)^{\gamma}.
\end{align*}Summarising the above, we have shown that if  $m>(k_{I}+1)n+s_{I}$ then 
\begin{equation}
\label{Bound2}
\mu(\widetilde{X}_{I^{k_I+1}(i_1,\ldots,i_{s_I})}\cap E_m)\asymp \mu(\widetilde{X}_{J})\mu(\widetilde{X}_{i_{|J|+1},\ldots,i_n})\varphi(m)^{\gamma}\varphi(n)^{\gamma}.
\end{equation}Adding together the bounds provided by \eqref{Bound1} and \eqref{Bound2}, we see that for any $m>n$ we have 
\begin{align*}
\mu(\widetilde{X}_{I^{k_{I}+1}(i_1,\ldots,i_{s_I})}\cap E_m)&\preceq \mu(\widetilde{X}_{J})\mu(\widetilde{X}_{i_{|J|+1},\ldots,i_n})\kappa^{m-n}\varphi(m)^{\gamma}\\
&+\mu(\widetilde{X}_{J})\mu(\widetilde{X}_{i_{|J|+1},\ldots,i_n}))\varphi(m)^{\gamma}\varphi(n)^{\gamma}.
\end{align*}This completes our proof.

\end{proof}

\begin{proposition}
	\label{Quasi independence prop}
	$$\sum_{n,m=|J|}^{Q}\mu(E_n\cap E_m)\preceq \mu(X_J)\left(\sum_{n=|J|}^{Q}\varphi(n)^{\gamma}+\left(\sum_{n=|J|}^{Q}\varphi(n)^{\gamma}\right)^2\right).$$
\end{proposition}
\begin{proof}
We have the following
\begin{align*}
\sum_{n,m=|J|}^{Q}\mu(E_n\cap E_m)&=\sum_{n=|J|}^{Q}\mu(E_n)+2\sum_{n=|J|}^{Q-1}\sum_{m=n+1}^{Q}\mu(E_n\cap E_m)\\
&=\sum_{n=|J|}^{Q}\mu(E_n)+2\sum_{n=|J|}^{Q-1}\sum_{\stackrel{I\in \D^n}{J\textrm{ is a prefix of }I}}\sum_{m=n+1}^{Q}\mu(\widetilde{X}_{I^{k_{I}+1}(i_1,\ldots,i_{s_{I}})}\cap E_m).
\end{align*} Applying Lemma \eqref{Near independence} to the above we obtain
\begin{align}
\label{Triple split}
\sum_{n,m=|J|}^{Q}\mu(E_n\cap E_m)&\preceq\sum_{n=|J|}^{Q}\mu(E_n)+\sum_{n=|J|}^{Q-1}\sum_{\stackrel{I\in \D^n}{J\textrm{ is a prefix of }I}}\sum_{m=n+1}^{Q}\mu(\widetilde{X}_{J})\mu(\widetilde{X}_{i_{|J|+1},\ldots,i_n})\kappa^{m-n}\varphi(m)^{\gamma}\\
&+\sum_{n=|J|}^{Q-1}\sum_{\stackrel{I\in \D^n}{J\textrm{ is a prefix of }I}}\sum_{m=n+1}^{Q}\mu(\widetilde{X}_{J})\mu(\widetilde{X}_{i_{|J|+1},\ldots,i_n}))\varphi(m)^{\gamma}\varphi(n)^{\gamma}\nonumber.
\end{align}We focus on the three terms on the right hand side of \eqref{Triple split} individually. By Lemma \ref{Divergence lemma} we know that the following holds for the first term
\begin{equation}
\label{BOUND1}
\sum_{n=|J|}^{Q}\mu(E_n)\asymp \mu(X_{J})\sum_{n=|J|}^{Q}\varphi(n)^{\gamma}.
\end{equation}Now let us focus on the second term in \eqref{Triple split}. We have 
\begin{align}
\label{BOUND2}
&\sum_{n=|J|}^{Q-1}\sum_{\stackrel{I\in \D^n}{J\textrm{ is a prefix of }I}}\sum_{m=n+1}^{Q}\mu(\widetilde{X}_{J})\mu(\widetilde{X}_{i_{|J|+1},\ldots,i_n})\kappa^{m-n}\varphi(m)^{\gamma}\nonumber\\
= &\mu(\widetilde{X}_{J})\sum_{n=|J|}^{Q-1}\sum_{\stackrel{I\in \D^n}{J\textrm{ is a prefix of }I}}\mu(\widetilde{X}_{i_{|J|+1},\ldots,i_n})\sum_{m=n+1}^{Q}\kappa^{m-n}\varphi(m)^{\gamma}\nonumber\\
= &\mu(\widetilde{X}_{J})\sum_{n=|J|}^{Q-1}\sum_{I'\in \D^{n-|J|}}\mu(\widetilde{X}_{I'})\sum_{m=n+1}^{Q}\kappa^{m-n}\varphi(m)^{\gamma}\nonumber\\
\stackrel{\eqref{Sum prefixes}}{\preceq}&\mu(\widetilde{X}_{J})\sum_{n=|J|}^{Q-1}\sum_{m=n+1}^{Q}\kappa^{m-n}\varphi(m)^{\gamma}\nonumber\\
=&\mu(\widetilde{X}_{J})\sum_{m=|J|+1}^Q\sum_{n=|J|}^{m-1}\kappa^{m-n}\varphi(m)^{\gamma}\nonumber\\
\stackrel{(\kappa<1)}{\preceq}&\mu(\widetilde{X}_{J})\sum_{m=|J|+1}^Q\varphi(m)^{\gamma}\nonumber\\
\stackrel{\eqref{Samemeasure}}{=}&\mu(X_{J})\sum_{m=|J|+1}^Q\varphi(m)^{\gamma}.
\end{align}It remains to bound the third term:
\begin{align}
\label{BOUND3}
&\sum_{n=|J|}^{Q-1}\sum_{\stackrel{I\in \D^n}{J\textrm{ is a prefix of }I}}\sum_{m=n+1}^{Q}\mu(\widetilde{X}_{J})\mu(\widetilde{X}_{i_{|J|+1},\ldots,i_n}))\varphi(m)^{\gamma}\varphi(n)^{\gamma}\nonumber\\
=&\mu(\widetilde{X}_J)\sum_{n=|J|}^{Q-1}\varphi(n)^{\gamma}\sum_{\stackrel{I\in \D^n}{J\textrm{ is a prefix of }I}}\mu(\widetilde{X}_{i_{|J|+1},\ldots,i_n}))\sum_{m=n+1}^{Q}\varphi(m)^{\gamma}\nonumber\\
=&\mu(\widetilde{X}_J)\sum_{n=|J|}^{Q-1}\varphi(n)^{\gamma}\sum_{I'\in \D^{n-|J|}}\mu(\widetilde{X}_{I'})\sum_{m=n+1}^{Q}\varphi(m)^{\gamma}\nonumber\\
\stackrel{\eqref{Sum prefixes}}{\preceq}&\mu(\widetilde{X}_J)\sum_{n=|J|}^{Q-1}\varphi(n)^{\gamma}\sum_{m=n+1}^{Q}\varphi(m)^{\gamma}\nonumber\\
\stackrel{\eqref{Samemeasure}}{=}&\mu(X_J)\sum_{n=|J|}^{Q-1}\varphi(n)^{\gamma}\sum_{m=n+1}^{Q}\varphi(m)^{\gamma}\nonumber\\
\preceq&\mu(X_J)\left(\sum_{n=|J|}^{Q}\varphi(n)^{\gamma}\right)^2.
\end{align}
Collecting the bounds provided by \eqref{BOUND1}, \eqref{BOUND2}, and \eqref{BOUND3}, and substituting them into \eqref{Triple split} completes our proof. 
\end{proof}

With Proposition \ref{Quasi independence prop} we can now complete our proof of Theorem \ref{Main theorem}.

\begin{proof}[Proof of Theorem \ref{Main theorem}.2]
By \eqref{Sufficestoshow} to prove the divergence part of Theorem \ref{Main theorem} it suffices to show that there exists $c>0$ such that 
\begin{equation}
\label{Proofsufficestoshow}
\mu(\limsup_{n\to\infty} E_n)>c\mu(X_J).
\end{equation} 
By Lemma \ref{Divergence lemma} we know that $\sum_{n=|J|}^{\infty}\mu(E_n)=\infty$. Therefore Lemma \ref{Erdos lemma} and Lemma \ref{Divergence lemma} combined tell us that 
\begin{equation}
\label{nearlydone}\mu(\limsup_{n\to\infty} E_n)\geq \limsup_{Q\to\infty}\frac{\left(\sum_{n=|J|}^{Q}\mu(E_n)\right)^2}{\sum_{n,m=|J|}^{Q}\mu(E_n\cap E_m)}\succeq \limsup_{Q\to\infty}\frac{\mu(\widetilde{X}_J)^2\left(\sum_{n=|J|}^{Q}\varphi(n)^{\gamma}\right)^2}{\sum_{n,m=|J|}^{Q}\mu(E_n\cap E_m)}.
\end{equation}Since $\sum_{n=1}^{\infty}\varphi(n)^{\gamma}=\infty$, we know that for any $Q$ sufficiently large we have 
\begin{equation*}
\label{Square dominates}
\sum_{n=|J|}^{Q}\varphi(n)^{\gamma}<\left(\sum_{n=|J|}^{Q}\varphi(n)^{\gamma}\right)^2.
\end{equation*}
Combining this observation with \eqref{nearlydone} and Proposition \ref{Quasi independence prop}, it follows that $$\mu(\limsup_{n\to\infty} E_n)\succeq \mu(X_J),$$ i.e. there exists $c>0$ such that $\mu(\limsup_{n\to\infty} E_n)>c\mu(X_J).$ So we have shown that \eqref{Proofsufficestoshow} holds. This completes our proof.
\end{proof}

\noindent \textbf{Acknowledgements.} This second author was supported by the University of Warwick Undergraduate Research Support Scheme. The authors would like to thank the anonymous referee whose comments allowed us to strengthen Theorem \ref{Main theorem}.

\end{document}